\documentclass[11pt]{amsart}
\usepackage{amsmath,amssymb,amsfonts}
\usepackage{enumerate}
\usepackage{amscd, amssymb, latexsym, amsmath, amscd}
\usepackage[all]{xy}
\usepackage{pb-diagram}
\usepackage{mathrsfs}
\usepackage{pictexwd,dcpic}
\usepackage{hyperref,url}

\usepackage{pb-diagram}
\theoremstyle{plain}
\newtheorem{Thm}[equation]{Theorem}
\newtheorem{Conj}[equation]{Conjecture}
\newtheorem{Cor}[equation]{Corollary}
\newtheorem{Prop}[equation]{Proposition}
\newtheorem{Lem}[equation]{Lemma}
\numberwithin{equation}{section}

\theoremstyle{remark}
\newtheorem{Rmk}[equation]{Remark}

\newcommand{\Hom}{\operatorname{Hom}}
\newcommand{\Ext}{\operatorname{Ext}}
\newcommand{\Gal}{\operatorname{Gal}}
\newcommand{\GL}{\operatorname{GL}}

\newcommand{\SL}{\operatorname{SL}}

\newcommand{\Sp}{\operatorname{Sp}}

\newcommand{\Mp}{\operatorname{Mp}}

\newcommand{\U}{\operatorname{U}}

\newcommand{\Ind}{\operatorname{Ind}}

\newcommand{\As}{\operatorname{As}}
\newcommand{\Irr}{\operatorname{Irr}}

\newcommand{\Tr}{\operatorname{Tr}}
\newcommand{\Nm}{\operatorname{Nm}}

\newcommand{\Z}{\mathbb{Z}}

\newcommand{\bm}{\begin{multline*}}
\newcommand{\tu}{\end  {multline*}}

\newcommand{\disc}{{\rm disc}}

\def\calH{\mathcal{H}}

\def\calJ{\mathcal{J}}

\def\calR{\mathcal{R}}
\def\calS{\mathcal{S}}

\def\calX{\mathcal{X}}
\def\calY{\mathcal{Y}}

\title{Twisted Gan-Gross-Prasad conjecture for certain tempered L-packets}
\author{Rui Chen and Wee Teck Gan} 
\address{Department of Mathematics, National University of Singapore, Block S17, 10 Lower Kent Ridge Road, Singapore 119076. }
\email{matgwt@nus.edu.sg}
\address{Institute for Advanced Study in Mathmatics, ZheJiang University, East No.7 Building, Zijingang Campus, Hangzhou 310058, China }
\email{rchenmat@zju.edu.cn}
\begin{document}

\begin{abstract}
In this paper, we investigate the twisted GGP conjecture for certain tempered representations using the theta correspondence and establish some special cases, namely when the L-parameter of the unitary group is the sum of conjugate-dual characters of the appropriate sign.
\end{abstract}

\maketitle



\section{Problem, Conjecture and Results}
 In a recent paper \cite{TGGP}, a twisted version of the Gan-Gross-Prasad conjecture was formulated in the context of skew-Hermitian spaces and their associated unitary groups over local and global fields.  Some evidences were provided in \cite{TGGP} for  the local twisted conjecture, such as in low rank situations and for unitary principal series representations. The purpose of this paper is to provide further affirmative evidences, by establishing the local conjecture for a family of tempered L-packets of unitary groups using the technique of theta correspondence. Let us recall the setup and conjecture of \cite{TGGP} in greater precision and formulate our main result.

\vskip 10pt

\subsection{Biquadratic extension}
Let $F$ be a non-Archimedean local field of characteristic $0$, and $E \neq K$ two distinct quadratic field extensions of $F$. Let $L=E \otimes_F K$, so that $L$ is a biquadratic extension of $F$. We thus have the picture:
\[
\begindc{\commdiag}[80]
\obj(0,4)[a]{$L$}
\obj(-8,0)[b]{$K$}
\obj(8,-0)[c]{$E$}
\obj(0,-4)[d]{$F$}
\mor{b}{a}{$\sigma$}[\atleft,\solidline]
\mor{d}{b}{$\tau$}[\atleft,\solidline]
\mor{a}{c}{$\tau$}[\atleft,\solidline]
\mor{c}{d}{$\sigma$}[\atleft,\solidline]
\enddc
\]
In particular, we have set:
\[
    \Gal(E/F) \simeq \Gal(L/K) \simeq \langle\sigma\rangle, \quad \textit{and} \quad \Gal(K/F) \simeq \Gal(L/E) \simeq \langle\tau\rangle.
\]
We also fix an additive character $\psi_F$ of $F$, and set $\psi_K=\psi_F\circ\Tr_{K/F}$. In this paper, when we talk about Weil representations or theta correspondence, we always use these additive character $\psi_F$ or $\psi_K$ (see Section \ref{WR:Setup}).

\vskip 5pt

\subsection{Skew-Hermitian spaces}
Let $V$ be an $n$-dimensional skew-Hermitian space over $E$. There are exactly two such spaces, which are distinguished  by their sign 
\[
    \epsilon(V)=\omega_{E/F}(\delta^{-n}\cdot\disc V),
\]
where $\disc V=(-1)^{n(n-1)/2}\cdot\det V$, and $\delta$ is a fixed  trace zero element in $E^\times$. As observed in \cite[Lem. 6.1]{TGGP}, the scalar extension $V_K = V \otimes_F K$ is a distinguished  split skew-Hermitian space over $L$ whose isomorphism class is independent of the choice of $V$. In particular, if we continue to use the trace zero element $\delta\in L^\times$ to define the sign of $V_K$, then we always have $\epsilon(V_K)=+1$.

\vskip 5pt

\subsection{Twisted GGP problem}
We come now to the restriction problem to be studied. For the skew-Hermitian space $V$ over $E$, we have the Weil representation $\omega_{V,\mu}$, where $\mu$ is a conjugate-symplectic character of $E^\times$. Then we are interested in determining
\[
    m_{V}(\pi,\mu) = \dim \Hom_{\U(V)}(\pi, \omega_{V,\mu}) \quad \textit{for } \pi\in \Irr\left(\U(V_K)\right).
\]
Here is the main local conjecture for the twisted GGP problem:

\vskip 5pt

\begin{Conj}\label{C-TGGP}
~
\begin{enumerate}
    \item For each $\pi\in \Irr\left(\U(V_K)\right)$, $m_{V}(\pi,\mu) \leq 1$. \\

    \item Let $M$ be a generic L-parameter of $\U(V_K)$ with associated L-packet $\Pi_M$. Then
        \[
            \sum_V \sum_{\pi\in\Pi_M} m_{V}(\pi,\mu) =1,
        \]
        where the first sum runs over the two skew-Hermitian spaces over $E$ of dimension $n$, and the second runs over the L-packet $\Pi_M$. \\

    \item The unique $V_0$ which has non-zero contribution to the sum in $(2)$ is characterized by 
        \[
            \epsilon(V_0)=\epsilon\left(\frac{1}{2}, \As^+_{L/E}(M)\otimes\mu^{-1}, \psi_{E,\delta}\right)\cdot\omega_{K/F}\left(\delta^2\right)^{n(n-1)/2},
        \]
        where $\delta$ is the fixed trace zero element in $E^\times$ (used in the definition of $\epsilon(V_0)$), and $\psi_{E,\delta}=\psi_F(\Tr_{E/F}(\delta\cdot~))$. \\

    \item The unique $\pi\in \Pi_M$ which has non-zero contribution to the sum in $(2)$ corresponds via the LLC (with respect to the Whittaker datum of $\U(V_K)$ associated to $\psi_K$) to the character of local component group $A_M=\prod_{i\in I}\Z/2\Z\cdot a_i$ given by:
        \begin{align*}
            \eta(a_i) &= \epsilon \left(\frac{1}{2}, \Ind_L^E\left(^{\tau}M_i \otimes (M/M_i)\right)\cdot\mu^{-1}, \psi_{E,\delta}\right)\\
            &= \epsilon\left(\frac{1}{2}, [\As(M_i)+\As(M)+\As(M/M_i)]\cdot\mu^{-1}, \psi_{E,\delta}\right),
        \end{align*}
        where $M_i$ is the irreducible constituent of $M$ corresponding to $a_i\in A_M$.
\end{enumerate}
\end{Conj}

\vskip 5pt
We remark that \cite{TGGP} also formulated a conjecture in the case $E = K$ and showed that, in this case, the conjecture can be reduced to the case of discrete series representations of $\U(V_K) \simeq \GL(V)$. However, we do not deal with the case $E = K$ in this paper. 

\vskip 5pt
In \cite[Sect. 7]{TGGP}, the three authors have proved that:
\begin{Thm}  \label{T-TGGP}
~
\begin{enumerate}
\item Conjecture \ref{C-TGGP} holds if $n\leq 2$.

\vskip 5pt

\item Conjecture \ref{C-TGGP}(1)-(3) hold for unitary principal series representations (induced from the Borel subgroup), and (4) holds as well if the unitary principal series is irreducible. 
\end{enumerate}
\end{Thm}

\vskip 5pt
\subsection{\bf Main result}
Our main result is  the following theorem.
\vskip 5pt

\begin{Thm}\label{T:Main}
Let $M$ be a tempered L-parameter for $\U(V_K)$ of the form
\[
    M = M_1+\cdots +M_n 
\]
with each $M_i$ $1$-dimensional and conjugate self-dual of parity $(-1)^{n-1}$.
Then Conjecture \ref{C-TGGP} holds for $M$.
\end{Thm}

Note that though these tempered L-parameters $M$ are maximally reducible and hence not the most general in the $p$-adic case, they are the ones whose L-packets are of maximal size. Hence, in some sense, they provide the most stringent test for Conjecture \ref{C-TGGP}.   
An immediate corollary of our result is that we may complete Theorem \ref{T-TGGP}(2)  above: 
\vskip 5pt

\begin{Cor}
Conjecture \ref{C-TGGP} holds for the tempered L-packets associated with unitary principal series representations.
\end{Cor}
\vskip 5pt

\subsection{\bf Idea of proof}
The main tool for the proof of Theorem \ref{T:Main} is the theta correspondence. Using theta correspondence, we shall effectively show that Theorem \ref{T:Main} for the case $\dim V = n+1$ can be reduced to the case for $\dim V = n$. In this way, for the type of tempered L-parameters $M$ considered in Theorem \ref{T:Main}, we may use theta correspondence to successively strip off the irreducible summands $M_i$ one at a time and reduce the conjecture for such $M$'s to the case when $\dim V = 1$. In fact, since the conjecture has been shown for $\dim V \leq 2$, we could have formulated a slightly more general main result. We content ourselves with just the following corollary:
\vskip 5pt

\begin{Cor}
Conjecture \ref{C-TGGP} holds for all endsocopic tempered L-packets of $\U(V_K)$ when $\dim V = 3$. 
\end{Cor}
This is because all endoscopic tempered L-packets of $\U_3$ can be constructed by theta lifting from tempered L-packets of $\U_2$. 
\vskip 10pt

The rest of the paper is devoted to the proof of Theorem \ref{T:Main}.
In \S \ref{S:WR}, we study a local theta lift of a Weil representation of a unitary group to the edge of the stable range. The main point here is to show that the resulting big theta lift is irreducible. Then in \S \ref{S:Proof}, we show how the conjecture in dimension $n+1$ can be reduced to that in dimension $n$ by invoking two seesaw arguments.

\vskip 10pt

\section{Weil representations}\label{S:WR}
In this section, we examine the Weil representation $\omega_{V,\mu}$, and investigate its behaviour under the theta correspondence.

\vskip 5pt

\subsection{Local theta correspondence}\label{WR:Setup}
We first recall the basic setup of the local theta correspondence. Let $F\subset E$ be a quadratic extension of non-Archimedean local fields, $V$ an skew-Hermitian space of dimension $n$ and $W$ an Hermitian space of dimension $m$. We shall use the symbol $\calH$ (resp. $\calH'$) to denote the skew-Hermitian (resp. Hermitian) hyperbolic plane. \\

To consider the theta correspondence for the reductive dual pair $\U(V)\times \U(W)$, one requires some additional data:
\begin{itemize}
\item a non-trivial additive character $\psi_F$ of $F$;
\vskip 5pt

\item a pair of characters $\chi_V$ and $\chi_W$ of $E^\times$ such that 
\begin{equation*}
\chi_V\big|_{F^\times}=\omega_{E/F}^{\dim V}\quad\textit{and}\quad\chi_W\big|_{F^\times}=\omega_{E/F}^{\dim W}.
\end{equation*}
\vskip 5pt
\end{itemize}
To elaborate, the tensor product $V\otimes W$ has a natural symplectic form, which induces a natural map
\begin{equation*}
\U(V)\times \U(W)\longrightarrow \Sp(V\otimes W).
\end{equation*}
One has the metaplectic $S^1$-cover $\Mp(V\otimes W)$ of $\Sp(V\otimes W)$, and the character $\psi_F$ determines a Weil representation $\omega_{\psi_F}$ of $\Mp(V\otimes W)$. The datum $(\psi_F,\chi_V,\chi_W)$ then allows one to specify a splitting of the metaplectic cover over $\U(V)\times \U(W)$ following \cite{MR1286835}. Hence, we have a Weil representation $\omega=\omega_{V,W}$ of $\U(V)\times \U(W)$. \\

As explicated in \cite{MR1286835} and \cite{MR1327161}, the splitting over $\U(V)$ is determined by $(\psi_F,\chi_W)$, whereas that of $\U(W)$ by $(\psi_F,\chi_V)$. In particular, taking $W$ such that $\dim W=1$ and $\chi_W=\mu$ a conjugate symplectic character of $E^\times$, one gets a splitting over $\U(V)$ associated to $(\psi_F,\mu)$, and thus a Weil representation $\omega_{V,\mu}$ of $\U(V)$, which is the one appearing in the main conjecture.\\

Given an irreducible representation $\pi$ of $\U(V)$, the maximal $\pi$-isotypic quotient of $\omega$ is of the form 
\begin{equation*}
\Theta(\pi)\boxtimes\pi
\end{equation*}
for some smooth representation $\Theta(\pi)$ of $\U(W)$ of finite length. By the Howe duality \cite{MR1159105} \cite{MR3502978} \cite{MR3454380}, we have:
\begin{itemize}
\item The maximal semi-simple quotient $\theta(\pi)$ of $\Theta(\pi)$ is irreducible if $\Theta(\pi)$ is non-zero;
\item If $\pi_1\not\simeq\pi_2$ are two non-isomorphic irreducible smooth representations of $\U(V)$ such that both $\theta(\pi_1)$ and $\theta(\pi_2)$ are non-zero, then $\theta(\pi_1)\not\simeq\theta(\pi_2)$.
\end{itemize}

\vskip 5pt

\subsection{A refinement of Adams' conjecture} \label{SS:AC}
Next we give a description of the theta correspondence in terms of A-parameters. We fix a non-trivial additive character $\psi_F$ once and for all. Assume that
\[
    m=\dim W \geq n=\dim V \geq 1.
\]
Fix a pair of splitting characters $(\chi_V,\chi_W)$ and consider the theta correspondence between $\U(V) \times \U(W)$ with respect to it.

\vskip 5pt

Let $\psi$ be a local A-parameter of $\U(V)$. If we write it as a sum of irreducible subrepresentations
\[
    \psi=\sum_i \rho_i S_{a_i} \boxtimes S_{b_i},
\] 
we say that $\psi$ is of good parity if $\rho_i S_{a_i} \boxtimes S_{b_i}$ is conjugate self-dual of parity $(-1)^{n-1}$ for all $i$. Here we are following Atobe-Gan's notation \cite{MR3714507} on irreducible representations of the Weil-Deligne group $WD_E=W_E\times \SL_2$; we omit the tensor symbol between $\rho_i$ and $S_{a_i}$ to distinguish finite dimensional representations of the Weil-Deligne $\SL_2$ and the Arthur $\SL_2$.

\begin{Thm} \label{T:AC-P}
~
\begin{enumerate}
\item Assume that $\psi$ is of good parity and
\[
    m-n \geq \max_i \left\{b_i-a_i+1~\big|~\rho_i\simeq\chi_W\right\}.
\]
Let $\pi$ be an irreducible unitary representation in the local A-packet $\Pi_\psi(\U(V))$. Then the theta lift $\theta(\pi)$ of $\pi$ to $\U(W)$ lies in the local A-packet $\Pi_{\theta(\psi)}(\U(W))$ if it is non-zero, where 
\[
    \theta(\psi)=\psi\chi_W^{-1}\chi_V+\chi_V\boxtimes S_{m-n}.
\]~

\item
Moreover, if we further assume that
\[
    m-n > \max_i \left\{b_i+a_i-1~\big|~\rho_i\simeq\chi_W\right\},
\]
then $\theta(\pi)$ must be non-zero for any $\pi\in\Pi_\psi(\U(V))$.
\end{enumerate}
\end{Thm}

\begin{proof}
This is \cite[Thm. 5.2]{MR2906916}.

\end{proof}

\vskip 5pt

\begin{Rmk}~
\begin{enumerate}
\item There is a caveat here: M{\oe}glin's result \cite[Thm. 5.2]{MR2906916} is for the symplectic-orthogonal dual pair. If one assumes M{\oe}glin's explicit construction of A-packets for unitary groups (both quasi-split and non quasi-split), then M{\oe}glin's proof of \cite[Thm. 5.2]{MR2906916} should also work for unitary dual pairs. \\

\item In later proofs of our main result, we will only use A-packets of unitary groups in some special cases; those A-packets are the Zelevinsky-Aubert dual of some tempered L-packets. Since the LLC for unitary groups has been fully established (see \cite{MR3338302,kaletha2014endoscopic,MR3839702,MR4324358}), our main result in this paper is not conditional on the construction of A-packets for unitary groups.
\end{enumerate} 
\end{Rmk}

\vskip 5pt

Recall that for each local A-parameter $\psi$, the local A-packet $\Pi_\psi(\U(V))$ is also equipped with a map (depending on the choice of the additive character $\psi_F$)
\[
    \calJ:\Pi_\psi(\U(V)) \longrightarrow \Irr A_\psi,
\]
where $A_\psi$ is the component group associated to $\psi$. For example, if $\psi$ is a local A-parameter of good parity as above, then
\[
    A_\psi = \sum_j \Z/2\Z \, a_j
\]
is a free $\Z/2\Z$-module with a canonical basis $\{a_j\}_j$, where $j$ runs over a representative set of inequivalent subrepresentations of $\psi$.

\begin{Thm} \label{T:AC-L}
In the context of Theorem \ref{T:AC-P}(1), let $\pi\in\Pi_\psi(\U(V))$ and $\eta$ the character of $A_\psi$ associated to $\pi$. If the theta lift $\theta(\pi)$ is non-zero, then it corresponds to the character $\theta(\eta)$ of $A_{\theta(\psi)}$, where $\theta(\eta)$ can be uniquely determined as follows:
\begin{itemize}
\item if $n$ and $m$ are of different parity, then
\[ 
\theta(\eta)~\Big|_{A_\psi} =\eta;
\]~

\item if $n$ and $m$ are of the same parity, then
\[
\theta(\eta)(a_j)/\eta(a_j)=\epsilon \left(\frac{1}{2}, \psi_j\chi_W^{-1}, \psi_{E,\delta}\right),
\]
where $a_j\in A_\psi$ is the basis element corresponding to the irreducible summand $\psi_j$ of $\psi$.
\end{itemize}
\end{Thm}   

\begin{proof}
This can be proved as in \cite[Sect. 7.4]{MR3788848}. See also \cite[Cor. 7.4]{ThetaNAPack}.

\end{proof}

\vskip 5pt
\subsection{\bf A result of Atobe}

The following lemma, which is essentially due to Atobe, is useful to us in the later proofs. Let $\nu$ be the normalized absolute value of $E^\times$. 

\begin{Lem}\label{Length2}
Let $G_0=\U(V_0)$ be the unitary group associated to some Hermitian (or skew-Hermitian) space $V_0$, and $\psi_0$ an A-parameter of $G_0$. Suppoe that $\psi_0$ is elementary and trivial on the Weil-Deligne $\SL_2$. Let $\rho$ be an irreducible representation of $W_E$ and $x\in \frac{1}{2}\Z$ positive such that
\[
    \rho\boxtimes S_{2x-1}\subset\psi_0\quad \textit{and} \quad \rho\boxtimes S_{2x+1}\not\subset\psi_0.
\]
Then for any $\pi_0\in\Pi_{\psi_0}(G_0)$, we have a non-split exact sequence:
\[
    0 \longrightarrow \pi \longrightarrow \rho\nu^{-x}\rtimes \pi_0 \longrightarrow \pi' \longrightarrow 0,
\]
where $\pi$ is the unique irreducible subrepresentation and $\pi'$ is the unique irreducible quotient of $\rho\nu^{-x}\rtimes \pi_0$. In particular, the length of the induced representation $\rho\nu^{-x}\rtimes \pi_0$ is $2$.
\end{Lem} 

\vskip 5pt

\begin{proof}
Let $\phi_0=\widehat{\psi_0}$ be the Aubert dual of $\psi_0$, which is a discrete L-parameter for $G_0$. Then apply \cite[Lem. 5.1]{MR4055180} to $\widehat{\pi_0}\in\Pi_{\phi_0}(G_0)$.

\end{proof}

\vskip 5pt

\begin{Rmk}
In fact, Atobe only considered split odd orthogonal groups and symplectic groups in \cite{MR4055180}. But at least his Lemma 5.1 should be valid for unitary groups as well; the proof of the unitary case is totally the same as Atobe's proof of orthogonal/ symplectic case.
\end{Rmk}

\vskip 5pt

\subsection{Some local A-packets}
Now we use the Adams' conjecture to describe Weil representations. Let $E^1$ be the subgroup of $E^\times$ consists of norm $1$ elements. Let $\chi_0$ be a character of $E^1$ and $\chi$ the character of $E^\times$ obtained from $\chi_0$ by base change; we may regard $\chi$ as the L-parameter of the unitary group $E^1$ corresponding to $\chi_0$. We denote by $\omega_{V,\mu}[\chi]$ the maximal subrepresentation of $\omega_{V,\mu}$ such that the center of $\U(V)$ acts by $\chi_0$. When $n=1$, the representation $\omega_{V,\mu}[\chi]$ has been studied by \cite{MR876022} and \cite{MR1155235}. So we shall concentrate on the case $n \geq 2$.

\begin{Lem}\label{L:WR-L}
~
\begin{enumerate}
    \item If $n=2$ and $\chi=\mu^2$, then the representation $\omega_{V,\mu}[\chi]$ is non-zero only when the space $V$ is of sign $+1$, in which case $\omega_{V,\mu}[\chi]$ is the generic member (with respect to the generic datum defined by $\psi_F$) in the L-packet $\Pi_\phi(\U(V))$, where
    \[
        \phi=\mu+\mu.
    \]
    \vskip 5pt

    \item For any $n=\dim V\geq 2$, excluding the special case above, the representation $\omega_{V,\mu}[\chi]$ is non-zero, irreducible and unitary. It lies in the A-packet $\Pi_{\Psi}(\U(V))$, where 
\[
    \Psi=\chi\cdot\mu^{-n+1}+\mu\boxtimes S_{n-1}.
\]
The character $\eta\in\Irr A_\Psi$ associated to $\omega_{V,\mu}[\chi]$ is
\[
    \eta:(e_1, e_{n-1})\longmapsto
    \begin{cases}
    (1,\epsilon(V)) \quad &\textit{if $n$ is even,}\\
    \\
    \Big(\epsilon\left(\frac{1}{2},\chi\mu^{-n},\psi_{E,\delta}\right)\,,\,\epsilon(V)\epsilon\left(\frac{1}{2},\chi\mu^{-n},\psi_{E,\delta}\right)\Big) \quad &\textit{if $n$ is odd.}
    \end{cases}
\]
Here $e_1$ and $e_{n-1}$ are the basis elements of $A_{\Psi}$ corresponding to $\chi\cdot\mu^{-n+1}$ and $\mu\boxtimes S_{n-1}$ respectively.
\end{enumerate}

\end{Lem}

\vskip 5pt

\begin{proof}
Let $L_1$ be the $1$-dimensional Hermitian space associated to $1\in F^\times$. Let $\chi_V$ be a character of $E^\times$ such that $\chi_V~\big|_{F^\times}=\omega_{E/F}^n$, and $\Omega_{L_1,V}$ the Weil representation associated to $\U(L_1)\times \U(V)$ with respect to the splitting character $(\mu,\chi_V)$. Then we have 
\[
    \Omega_{L_1,V}~\Big|_{\U(V)}=\omega_{V,\mu}.
\]
Hence $\omega_{V,\mu}[\chi]$ can be regarded as the theta lift of the character $\chi\mu^{-n}\chi_V$. Thus our first assertion follows from Theorem \ref{T:AC-P}, and the second follows from Theorem \ref{T:AC-L}.

\end{proof}

\vskip 5pt

\subsection{Irreducibility of big theta lifts}
Finally we investigate the irreducibility of the big theta lift of $\omega_{V,\mu}[\chi]$. We shall work in a slightly more general setting. 

\vskip 5pt

We retain the notations of Section \ref{SS:AC}. From now on we assume that $m$ is even and $m\geq\max\{2n-2,n\}$. Let
\[
    \psi=\delta+\mu\boxtimes S_{n-1}
\]
be a local A-parameter of $\U(V)$, where $\delta$ and $\mu$ are conjugate self-dual characters of appropriate parities. Our goal is to show the following.

\begin{Thm}\label{ThmBigThetaEnhanced}
For any $\pi\in\Pi_\psi(\U(V))$, the big theta lift $\Theta(\pi)$ to $\U(W)$ is irreducible if it is non-zero. Moreover, we have
\[
    \Ext_{\U(V)}^i(\Omega,\pi)_{sm}=0
\]
for all $i>0$. Here $\Omega$ is the Weil representation associated to $\U(V)\times\U(W)$, and the subscript ``$sm$'' stands for taking the $\U(W)$-smooth vectors. 
\end{Thm}

\vskip 5pt

\begin{Rmk}
Although in this theorem we do not assert the non-vanishing of $\Theta(\pi)$, in the range we are considering (i.e. $m\geq 2n-2$ and $n\geq 2$), we are almost always in the situation of Theorem \ref{T:AC-P}(2), except for the following low rank cases:
\begin{itemize}
    \item $n=2$ and $m=2$ (this case will not be used in the proof of our main theorem);
    \vskip 5pt

    \item $n=3$, $m=4$ and $\delta=\chi_W$.
\end{itemize}
\end{Rmk}

\vskip 5pt

We shall prove this theorem by induction on the dimension of $V$. Let $x_n=-n/2+1$. According to M{\oe}glin \cite[Sect. 2.4]{MR2209850}, we know that:
\begin{Lem} \label{L:PJM}
Assume that $n\geq 3$, and let $\pi\in\Pi_\psi(\U(V))$. If $\pi$ is not supercuspidal, then there exists a unique $\pi_0\in\Pi_{\psi_0}(\U(V_0))$, such that 
\[
    \pi\hookrightarrow\mu\nu^{x_n}\rtimes\pi_0.
\]
Here
\[
    \psi_0=\delta+\mu\boxtimes S_{n-3},
\] 
and $V_0$ is a subspace of $V$ such that $V\simeq V_0 \oplus \calH$.
\end{Lem}

\vskip 5pt

Using this fact, we now do the induction step.
\begin{Prop}\label{P:ThetaEnhancedInd}
In the context of Lemma \ref{L:PJM}, assume that
\[
    \pi\hookrightarrow\mu\nu^{x_n}\rtimes\pi_0.
\]
Let $W_0$ be a subspace of $W$ such that $W\simeq W_0 \oplus \calH'$. Consider the theta correspondence of $\U(V_0)\times\U(W_0)$ (with respect to the same splitting characters). Then if Theorem \ref{ThmBigThetaEnhanced} holds for $\pi_0$, it also holds for $\pi$.
\end{Prop}

\vskip 5pt

\begin{Rmk}
In the setting of Theorem \ref{ThmBigThetaEnhanced}, we have assumed that $m\geq 2n-2$. Note that the dimensions of $V_0$ and $W_0$ also satisfy this inequality. Hence it makes sense to talk about Theorem \ref{ThmBigThetaEnhanced} for $\pi_0$.
\end{Rmk}

\vskip 5pt

\begin{proof}[Proof of Proposition \ref{P:ThetaEnhancedInd}]
Let $P$ be the standard parabolic subgroup of $\U(V)$ with Levi component $\GL_1\times\U(V_0)$. For $i\geq 0$, by the (derived version of) Frobenius reciprocity, we have
\begin{align*}
\Ext^i_{\U(V)}(\Omega,\mu\nu^{x_n}\rtimes\pi_0) = \Ext^i_{\GL_1\times \U(V_0)}(R_P\Omega,\mu\nu^{x_n}\boxtimes\pi_0),
\end{align*}
where $R_P$ is the normalized Jacquet module along $P$. To compute the RHS of above, one can appeal to the Kudla's filtration. There is a two-step filtration on $R_P\Omega$: 
\[
    R_P\Omega=R^0 \supset R^1 \supset R^2=0,
\]
whose successive quotient $J^a=R^a/R^{a+1}$ can be described as follows:
\[
    J^0=\chi_W\nu^{\frac{m-n+1}{2}}\boxtimes \Omega_0,
\]
and 
\[
    J^1=\Ind_{\GL_1\times\U(V_0)\times Q}^{\GL_1\times\U(V_0)\times \U(W)}\left(\calS(E^\times)\boxtimes \Omega_{00}\right).
\]
Here: 
\begin{itemize}
\item $\Omega_0$ is the Weil representation associated to $\U(V_0)\times\U(W)$; 
\vskip 5pt

\item $Q$ is a maximal parabolic subgroup of $\U(W)$ stabilizing an isotropic line of $W$; the Levi subgroup of $Q$ is isomorphic to $\GL_1\times\U(W_0)$; 
\vskip 5pt

\item $\calS(E^\times)$ is the space of Schwartz functions on $E^\times$, equipped with the natural action of two copies of $\GL_1$ (twisted by the splitting characters);
\vskip 5pt

\item $\Omega_{00}$ is the Weil representation associated to $\U(V_0)\times\U(W_0)$;
\vskip 5pt
\end{itemize}
Applying the functor $\Hom_{\GL_1\times\U(V_0)}(\cdot,\mu\nu^{x_n}\boxtimes\pi_0)$ to the short exact sequence
\[
    0\longrightarrow J^1 \longrightarrow R_P\Omega \longrightarrow J^0  \longrightarrow 0,
\]
we get a long exact sequence
\begin{align*}
    \cdots\longrightarrow \Ext^i(J^0,\mu\nu^{x_n}\boxtimes\pi_0)_{sm} \longrightarrow \Ext^i(R_P\Omega &,\mu\nu^{x_n}\boxtimes\pi_0)_{sm}\\ 
    \longrightarrow \Ext^i(J^1 &,\mu\nu^{x_n}\boxtimes\pi_0)_{sm} \longrightarrow \Ext^{i+1}(J^0,\mu\nu^{x_n}\boxtimes\pi_0)_{sm} \longrightarrow\cdots 
\end{align*}
Here for the exactness of taking $\U(W)$-smooth vectors, one may refer to \cite[Lem. 5.14, Lem. 7.4]{MR3753906}. Since by our assumptions ${x_n} \ne \frac{m-n+1}{2}$, we know that 
\[
    \Ext^i(J^0,\mu\nu^{x_n}\boxtimes\pi_0)=0
\]
for all $i\geq 0$. This implies that
\begin{align*}
\Ext^i(R_P\Omega,\mu\nu^{x_n}\boxtimes\pi_0)_{sm} \simeq & \Ext^i(J^1,\mu\nu^{x_n}\boxtimes\pi_0)_{sm}\\
    =&\left(\mu'\right)^c\nu^{x_n}\rtimes \Ext^i(\Omega_{00},\pi_0)_{sm},
\end{align*}
where $\mu'=\mu\chi_W^{-1}\chi_V$. Here in the last equality, we have made use of (an $\Ext$-version of) \cite[Lem. 2.6]{MR3753906} and the K\"unneth formula \cite[Lem. 3.3]{MR3753906}. In particular, we get 
\[
    \Hom_{\U(V)}(\Omega,\mu\nu^{x_n}\rtimes\pi_0)_{sm}\simeq\left(\mu'\right)^c\nu^{x_n}\rtimes \Theta(\pi_0)^\vee
\]
and 
\begin{equation*}\label{Acyclic-X}\tag{$\spadesuit$}
    \Ext^i_{\U(V)}(\Omega,\mu\nu^{x_n}\rtimes\pi_0)_{sm} = 0 \quad \textit{for $i>0$}
\end{equation*}
by our induction hypothesis. Similarly, we also have $-{x_n} \ne \frac{m-n+1}{2}$. The same argument gives that
\begin{equation*}\label{Acyclic-Y}\tag{$\clubsuit$}
    \Ext^i_{\U(V)}(\Omega,\mu\nu^{-x_n}\rtimes\pi_0)_{sm} = 0 \quad \textit{for $i>0$}.
\end{equation*}

\vskip 5pt

Now note that Lemma \ref{Length2} asserts that $\mu\nu^{x_n}\rtimes\pi_0$ is of length $2$. Let $\pi'$ be the unique irreducible quotient of $\mu\nu^{x_n}\rtimes\pi_0$, the sequence
\[
    0\longrightarrow \pi \longrightarrow \mu\nu^{x_n}\rtimes\pi_0 \longrightarrow \pi'  \longrightarrow 0
\]
is exact. Applying the functor $\Hom_{\U(V)}(\Omega,\cdot)$ to this short exact sequence and taking $\U(W)$-smooth vectors, we get
\begin{align*}\tag{$\heartsuit$}\label{L.E.S.Heart}
0\rightarrow \Hom(\Omega,\pi)_{sm}\rightarrow \Hom(\Omega,\mu\nu^{x_n}\rtimes\pi_0)_{sm}&\rightarrow \Hom(\Omega,\pi')_{sm}\rightarrow \Ext^1(\Omega,\pi)_{sm} \rightarrow  \\
\cdots\rightarrow\Ext^i(\Omega,\mu\nu^{x_n}\rtimes\pi_0)_{sm}\rightarrow\Ext^i(\Omega,\pi')_{sm}&\rightarrow\Ext^{i+1}(\Omega,\pi)_{sm}\rightarrow\Ext^{i+1}(\Omega,\mu\nu^{x_n}\rtimes\pi_0)_{sm}\rightarrow\cdots
\end{align*}
It follows from (\ref{Acyclic-X}) that
\[
    \Ext^i(\Omega,\pi')_{sm}\simeq\Ext^{i+1}(\Omega,\pi)_{sm} \quad \textit{for $i>0$}.
\]
Dually, apply both the contragredient and MVW-involution to $\mu\nu^{x_n}\rtimes\pi_0$, we get a dualized short exact sequence
\[
    0\longrightarrow \pi' \longrightarrow \mu\nu^{-x_n}\rtimes\pi_0 \longrightarrow \pi \longrightarrow 0.
\]
Similar to the argument above, this short exact sequence leads to a long exact sequence 
\begin{align*}\tag{$\diamondsuit$}\label{L.E.S.Diamond}
0\rightarrow \Hom(\Omega,\pi')_{sm}\rightarrow \Hom(\Omega,\mu\nu^{-x_n}\rtimes\pi_0)_{sm}&\rightarrow \Hom(\Omega,\pi)_{sm}\rightarrow \Ext^1(\Omega,\pi')_{sm} \rightarrow  \\
\cdots\rightarrow\Ext^i(\Omega,\mu\nu^{-x_n}\rtimes\pi_0)_{sm}\rightarrow\Ext^i(\Omega,\pi)_{sm}&\rightarrow\Ext^{i+1}(\Omega,\pi')_{sm}\rightarrow\Ext^{i+1}(\Omega,\mu\nu^{-x_n}\rtimes\pi_0)_{sm}\rightarrow\cdots
\end{align*}
which combining with (\ref{Acyclic-Y}) similarly implies that
\[
    \Ext^i(\Omega,\pi)_{sm}\simeq\Ext^{i+1}(\Omega,\pi')_{sm} \quad \textit{for $i>0$}.
\]
Playing ``Ping-Pong'', one can see that $\Ext^i(\Omega,\pi)$ is periodic:
\[
    \Ext^i(\Omega,\pi)_{sm} \simeq \Ext^{i+2}(\Omega,\pi)_{sm} \quad \textit{for $i>0$}.
\]
Since the higher extensions vanish when the degree is sufficiently large \cite[Pg. 98, Sect. 4.2]{BNote}, these groups $\Ext^i(\Omega,\pi)_{sm}$ must vanish for all $i>0$ with no other choice. The same reason also gives the vanishing of higher extensions of $\pi'$.\\

Suppose that $\Theta(\pi)\ne 0$. It remains to show that $\Theta(\pi)$ is irreducible. Thanks to the vanishing of higher extensions, we deduce from the long exact sequence (\ref{L.E.S.Heart}) that 
\[
    0\longrightarrow \Theta(\pi)^\vee \longrightarrow \left(\mu'\right)^c\nu^{x_n}\rtimes \Theta(\pi_0)^\vee \longrightarrow \Theta(\pi')^\vee  \longrightarrow 0
\]
is exact. In particular, $\Theta(\pi_0)$ must be non-zero, hence irreducible by our induction hypothesis. It then follows from Theorem \ref{T:AC-P} that $\Theta(\pi_0)$ lies in $\Pi_{\theta(\psi_0)}(\U(W_0))$, where
\[
    \theta(\psi_0)=\delta\chi_W^{-1}\chi_V+\mu\chi_W^{-1}\chi_V\boxtimes S_{n-3}+\chi_V\boxtimes S_{m-n}.
\]
Now we claim that the induced representation $\left(\mu'\right)^c\nu^{x_n}\rtimes \Theta(\pi_0)^\vee$ is of length $2$, and the two subquotients are non-isomorphic. Indeed, if $m-n=1$ and $\delta=\chi_W$, one can easily check this by hand. Otherwise note that:
\begin{itemize}
    \item $\theta(\psi_0)$ is elementary and trivial on the Weil-Deligne $\SL_2$;
    \vskip 5pt

    \item $\mu'\boxtimes S_{-2x_n-1}\subset\theta(\psi_0)$ but $\mu'\boxtimes S_{-2x_n+1}\not\subset\theta(\psi_0)$.
\end{itemize}
\vskip 5pt
In short, we are again in a situation such that we can appeal to Lemma \ref{Length2}, from which our claim follows. Therefore it suffices to check that $\Theta(\pi') \ne 0$. We shall argue by contradiction to show this. Suppose on the contrary that $\Theta(\pi') = 0$. Then on the one hand we have
\[
    \Theta(\pi)^\vee \simeq \left(\mu'\right)^c\nu^{x_n}\rtimes \Theta(\pi_0)^\vee,
\]
which implies that $\left(\mu'\right)^c\nu^{x_n}\rtimes \Theta(\pi_0)^\vee$ has socle $\theta(\pi)^\vee$. On the other hand, we also deduce from the long exact sequence (\ref{L.E.S.Diamond}) that 
\[
    0\longrightarrow \Theta(\pi')^\vee \longrightarrow \left(\mu'\right)^c\nu^{-x_n}\rtimes \Theta(\pi_0)^\vee \longrightarrow \Theta(\pi)^\vee  \longrightarrow 0
\]
is exact. Since we had assumed that $\Theta(\pi') = 0$, this exact sequence implies that
\[
    \left(\mu'\right)^c\nu^{-x_n}\rtimes \Theta(\pi_0)^\vee \simeq \Theta(\pi)^\vee.
\]
Applying both the contragredient and the MVW-involution, we get
\[
    \Theta(\pi)^{MVW} \simeq \left(\mu'\right)^c\nu^{x_n}\rtimes \Theta(\pi_0)^\vee,
\]
which implies that $\left(\mu'\right)^c\nu^{x_n}\rtimes \Theta(\pi_0)^\vee$ also has cosocle $\theta(\pi)^\vee$. This contradicts to our claim. Thus $\Theta(\pi') \ne 0$ as desired.

\end{proof}

\vskip 5pt

\begin{Cor}
Theorem \ref{ThmBigThetaEnhanced} holds.
\end{Cor}

\vskip 5pt

\begin{proof}
By using the previous lemma, we can use reduce Theorem \ref{ThmBigThetaEnhanced} to the case that $\pi$ is supercuspidal, or to the case that $n=0$. In the supercuspidal case:
\begin{itemize}
    \item it is well-known that the big theta lift of a supercuspidal representation is irreducible;
    \vskip 5pt

    \item all higher extensions vanish since supercuspidal representations of a unitary group are compact.
\end{itemize}
\vskip 5pt
In the case that $n=0$, $\U(V)$ is trivial and the Weil representation is simple a character of $\U(W)$. Hence Theorem \ref{ThmBigThetaEnhanced} holds.

\end{proof}

\vskip 10pt

\section{Proof of the main result}  \label{S:Proof}

In this section, we shall prove the main result: Theorem \ref{T:Main}. We first note:
\begin{Lem}\label{L:CT}
Assume that Conjecture \ref{C-TGGP} holds for a tempered L-parameter $M$. Then for any conjugate orthogonal character $\calX$ of $L^\times$, Conjecture \ref{C-TGGP} also holds for the L-parameter $M\cdot\calX$.
\end{Lem}

\vskip 5pt

\begin{proof}
To see this, one simply notes that
\[
    m_V(\Pi,\mu)=m_V\Big(\Pi\cdot{\calX_0},\mu\left(\calX~\big|_{E^\times}\right)\Big),
\]
where ${\calX_0}$ is the character of $L^1$ whose base change to $L^\times$ is $\calX$.

\end{proof}

Let $n \geq 2$ be an integer, and $V$ an $(n+1)$-dimensional skew-Hermitian space over $E$. We shall start with an L-parameter of the form
\[
    M=M_0 + M_1,
\] 
where $M_1$ is a conjugate self-dual character of parity $(-1)^n$.

\vskip 5pt

\subsection{Two seesaw diagrams: Uniqueness}

If there is an irreducible tempered representation $\Pi$ in the L-packet $\Pi_M$ corresponding to $\eta\in \Irr A_M$ such that
\[
    m_V(\Pi,\mu)\ne 0,
\]
we would like to lift $\Pi$ to some unitary group of $n$-variables to obtain some information. Let $\{a_i\}_{i=1}^r$ be a canonical basis of $A_M$, where each $a_i$ corresponds to some irreducible subrepresentation $M_i$ of $M$ (so $a_1$ corresponds to $M_1$). We set $\epsilon=\eta(a_1)$ and $W$ the unique $n$-dimensional Hermitian space over $L$ of sign $\epsilon$. Let $\left(\calX_V,\calX_W\right)$ be a pair of characters of $L^\times$, such that 
\[
    \calX_V~\big|_{K^\times}=\omega_{L/K}^{n+1} \quad \textit{and} \quad \calX_W=M_{1}.
\]
Then one can consider the theta correspondence between $\U(V_K)\times \U(W)$ with respect to the splitting character $\left(\calX_V,\calX_W\right)$. By \cite[Sect. 4.6(P2)]{MR3573972}, one knows that there is an irreducible tempered representation $\Sigma$ of $\U(W)$, such that
\[
    \Pi=\Theta(\Sigma)
\]
is the big theta lift of $\Sigma$. Indeed, one knows that $\Sigma$ has the L-parameter
\[
    \theta(M)=M_0\cdot\calX_W^{-1}\calX_V,
\]
and corresponds to the character $\theta(\eta) = \eta~\big|_{A_{\theta(M)}}$. Consider the following seesaw diagram:
\begin{equation*}\label{SeeSaw-1}\tag{$\natural.1$}
\begindc{\commdiag}[8]
\obj(-60,40)[aa]{$\U\left(\calR W\right)$}
\obj(-51,34)[a]{$~$}
\obj(-60,-40)[b]{$\U\left(W\right)$}
\obj(60,40)[c]{$\U\left(V_K\right)$}
\obj(60,-40)[d]{$\U(V)$}
\mor{aa}{b}{}[\atleft,\solidline]
\mor{a}{d}{}[\atleft,\solidline]
\mor{c}{d}{}[\atleft,\solidline]
\mor{c}{b}{}[\atleft,\solidline]
\enddc
\qquad \qquad
\begindc{\commdiag}[8]
\obj(-60,40)[aa]{$\Lambda=\Theta(\omega_{V,\mu}[\chi])$}
\obj(-51,34)[a]{$~$}
\obj(-60,-40)[bb]{$\Sigma$}
\obj(-54,-36)[b]{$~$}
\obj(60,40)[c]{$\Pi=\Theta(\Sigma)$}
\obj(51,34)[cc]{$~$}
\obj(60,-40)[d]{$\omega_{V,\mu}[\chi]$}
\mor{aa}{bb}{}[\atleft,\solidline]
\mor{a}{d}{}[\atleft,\solidline]
\mor{c}{d}{}[\atleft,\solidline]
\mor{cc}{b}{}[\atleft,\solidline]
\enddc
\end{equation*}
where:
\begin{itemize}
    \item $\calR W$ is the restriction of scalar of $W$ to $E$; \\

    \item the theta correspondence between $\U(V_K)\times\U(W)$ is with respect to some splitting characters $(\calX_V,\calX_W)$; \\

    \item the theta correspondence between $\U(V)\times\U\left(\calR W\right)$ is with respect to some splitting characters $(\chi_V,\chi_W)$; \\

    \item to make use of this seesaw diagram, we choose these splitting characters so that:
        \[
            \calX_V=\chi_V\circ\Nm_{L/E} \quad \textit{and} \quad \chi_W=\calX_W~\big|_{E^\times};
        \]
        \vskip 5pt

    \item $\chi$ is the L-parameter of the central character of the restriction of $\Pi$ to $\U(V)$, i.e.  
    \[
        \chi=\det(M)~\big|_{E^\times}.
    \]
\end{itemize}
\vskip 5pt
Then by the seesaw identity, we get
\begin{equation*}\label{E:SS1}\tag{$\maltese.1$}
    m_V\left(\Pi,\mu\right) = \dim \Hom_{\U(W)} \left(\Lambda, \Sigma\right).
\end{equation*}
In particular, $\Lambda$ is non-zero. By Lemma \ref{L:WR-L}, Theorem \ref{T:AC-P} and Theorem \ref{ThmBigThetaEnhanced}, we know that:
\begin{itemize}
    \item $\omega_{V,\mu}[\chi]$ lies in the A-packet $\Pi_{\Psi_{M,\mu}}(\U(V))$, where 
        \[
            \Psi_{M,\mu}=\chi\cdot\mu^{-n}+\mu\boxtimes S_n;
        \]

    \item $\Lambda$ is an irreducible unitary representation lies in the A-packet $\Pi_{\theta(\Psi_{M,\mu})}\left(\U\left(\calR W\right)\right)$, where
        \begin{align*}
            \theta\left(\Psi_{M,\mu}\right) &= \Psi_{M,\mu}\cdot\chi_W^{-1}\chi_V+\chi_V\boxtimes S_{n-1} \\
            &= \chi\cdot\mu^{-n}\cdot\chi_W^{-1}\chi_V + \chi_V\boxtimes S_{n-1} + \mu\cdot\chi_W^{-1}\chi_V\boxtimes S_n.
        \end{align*}
\end{itemize}

\vskip 10pt

To compute the RHS of equality (\ref{E:SS1}), we shall use another seesaw diagram:
\begin{equation*}\label{SeeSaw-2}\tag{$\natural.2$}
\begindc{\commdiag}[8]
\obj(60,40)[aa]{$\U\left(\calR W\right)$}
\obj(54,36)[a]{$~$}
\obj(60,-40)[b]{$\U\left(W\right)$}
\obj(-60,40)[c]{$\U\left(V^\flat_K\right)$}
\obj(-60,-40)[d]{$\U(V^\flat)$}
\mor{aa}{b}{}[\atleft,\solidline]
\mor{a}{d}{}[\atleft,\solidline]
\mor{c}{d}{}[\atleft,\solidline]
\mor{c}{b}{}[\atleft,\solidline]
\enddc
\qquad \qquad
\begindc{\commdiag}[8]
\obj(60,40)[a]{${\Theta(\omega)}$}
\obj(60,-40)[b]{$\Sigma$}
\obj(-60,40)[cc]{$\Pi^\flat=\Theta(\Sigma)$}
\obj(-54,36)[c]{$~$}
\obj(-60,-40)[d]{$\omega$}
\mor{a}{b}{}[\atleft,\solidline]
\mor{a}{d}{}[\atleft,\solidline]
\mor{cc}{d}{}[\atleft,\solidline]
\mor{c}{b}{}[\atleft,\solidline]
\enddc
\end{equation*}
where:
\begin{itemize}
    \item $V^\flat$ is an $n$-dimensional skew-Hermitian space over $E$ which will be suitably chosen later, and $V^\flat_K$ is its scalar extension to $L$;\\

    \item the theta correspondence between $\U(V^\flat)\times\U\left(\calR W\right)$ is with respect to some splitting characters $(\chi_{V^\flat},\chi'_W)$; \\

    \item the theta correspondence between $\U(V^\flat_K)\times\U(W)$ is with respect to some splitting characters $(\calX_{V^\flat},\calX'_W)$; \\

    \item to make use of this seesaw diagram, we choose these splitting characters so that:
        \[
            \calX_{V^\flat}=\chi_{V^\flat}\circ\Nm_{L/E} \quad \textit{and} \quad \chi'_W=\calX'_W~\big|_{E^\times};
        \]

    \item $\omega$ is some irreducible unitary representation of $\U(V^\flat)$ which will also be suitably chosen later. 
\end{itemize}
\vskip 5pt
We would like to choose these data appropriately such that $\omega$ is an irreducible constituent of some Weil representation, and $\Lambda=\Theta(\omega)$. To make this possible, we need to pick up these splitting characters very carefully. Let 
\[
    \chi_{V^\flat}=\mu\cdot\chi_W^{-1}\chi_V \quad \textit{and} \quad \calX'_W= M_1^{-1} \cdot \Upsilon,
\]
where $\Upsilon$ is a conjugate orthogonal character of $L^\times$ so that
\[
    \Upsilon ~\Big|_{E^\times} = \mu^2.
\]
It is not hard to see that such $\Upsilon$ exists. Then again by Theorem \ref{T:AC-P} one can see that $\omega$ (if exists) lies in the A-packet $\Pi_{\Psi^\flat}(\U(V^\flat))$, where 
        \[
            \Psi^\flat= \chi^\flat\cdot\mu^{-n+1} + \mu\boxtimes S_{n-1}, \quad \textit{with } \chi^\flat=\det(M/M_1)~\big|_{E^\times}.
        \]
Indeed, we have:

\begin{Prop}\label{P:AdamsWR}
Let $V^\flat$ be the $n$-dimensional skew-Hermitian space of sign
\begin{equation}\label{Eq-SignFlat}\tag{$\dagger$}
    \epsilon\left(V^\flat\right)=\begin{cases}
    +1 \quad & \textit{if $n=2$ and $\chi^\flat=\mu^2$},\\
    \\
    \epsilon(V)\cdot\epsilon\left(\calR W\right)\cdot\epsilon\left(\frac{1}{2},\As^+_{L/E}(M_{1})\cdot\mu^{-1}, \psi_{E,\delta}\right)\cdot\omega_{E/F}(-1)^n \quad & \textit{otherwise},
    \end{cases}
\end{equation}
and
\[
    \omega=\omega_{V^\flat,\mu}[\chi^\flat].
\]
Then $\Lambda$ is the (big) theta lift of $\omega$ to $\U(\calR W)$, i.e. $\Lambda=\Theta(\omega)$.
\end{Prop}

\vskip 5pt

\begin{proof}
We first check the special case that $n=2$ and $\chi^\flat=\mu^2$. So
\[
    \theta\left(\Psi_{M,\mu}\right)=\chi_V+\mu\cdot\chi_W^{-1}\chi_V\boxtimes S_2+\left(\chi_V^c\right)^\vee
\]
and 
\[
    \Lambda\subset \chi_V \rtimes (\chi_0\circ\det),
\]
where $\chi_0$ is the character of $E^1$ whose base change to $E^\times$ is $\mu\cdot\chi_W^{-1}\chi_V$. By the induction principle, one knows that the theta correspondence between $\U(V^\flat)\times\U(\calR W)$ defines a bijection 
\[
    \theta:\Pi_\phi(\U(V^\flat))\longrightarrow\Pi_{\theta\left(\Psi_{M,\mu}\right)}(\U(\calR W)),
\]
where $\phi=\mu+\mu$ is an L-parameter of $\U(V^\flat)$. Hence $\Lambda$ is the (big) theta lift of some $\omega\in \Pi_\phi(\U(V^\flat))$. To check that $\omega=\omega_{V^\flat,\mu}[\chi^\flat]$, one can compute the character $\eta^\flat\in\Irr A_\phi$ associated to $\omega$. Recall that $\Lambda$ is also the theta lift of $\omega_{V,\mu}[\chi]$. If we denote by $\eta\in \Irr A_{\Psi_{M,\mu}}$ and $\theta(\eta)\in\Irr A_{\theta(\Psi_{M,\mu})}$ the character associated to $\omega_{V,\mu}[\chi]$ and $\Lambda$ respectively, then by Lemma \ref{L:WR-L} and Theorem \ref{T:AC-L}, we have
\begin{align*}
    \theta(\eta)(a)&=\eta(a)=\epsilon\left(\frac{1}{2},\chi_W\mu^{-1},\psi_{E,\delta}\right).
\end{align*}
Here $a\in A_{\theta(\Psi_{M,\mu})}$ is the basis element corresponding to $\chi_V$, and we regard $A_\phi$ and $A_{\Psi_{M,\mu}}$ as subgroups of $A_{\theta(\Psi_{M,\mu})}$. Apply Theorem \ref{T:AC-L} again, we get
\begin{align*}
    \eta^\flat(a)&=\theta(\eta)(a)\cdot\epsilon\left(\frac{1}{2},\chi_V\cdot\chi_{V^\flat}^{-1},\psi_{E,\delta}\right)=1.
\end{align*}
This implies that $\omega=\omega_{V^\flat,\mu}[\chi^\flat]$.\\

Now excluding the special case above, we prove the general case. It would be convenient to consider the cases of odd and even $n$ separately. In the following, we check the case of odd $n$ in full details.\\

Let $e_1$, $e_{n-1}$ and $e_{n}$ be the basis elements of $A_{\theta\left(\Psi_{M,\mu}\right)}$ corresponding to $\chi\cdot\mu^{-n}\cdot\chi_W^{-1}\chi_V$, $\chi_V\boxtimes S_{n-1}$ and $\mu\cdot\chi_W^{-1}\chi_V\boxtimes S_n$ respectively. Then:

\begin{itemize}
    \item $A_{\Psi_{M,\mu}}$ can be regarded as the subgroup of $A_{\theta\left(\Psi_{M,\mu}\right)}$ generated by $e_1$ and $e_n$; \\

    \item $A_{\Psi^\flat}$ can be regarded as the subgroup of $A_{\theta\left(\Psi_{M,\mu}\right)}$ generated by $e_1$ and $e_{n-1}$.
\end{itemize}
\vskip 5pt
Recall that $\omega_{V,\mu}[\chi]\in\Pi_{\Psi_{M,\mu}}(\U(V))$ corresponds to the character $\nu_{n+1}$ of $A_{\Psi_{M,\mu}}$ such that
\[
    \nu_{n+1}: (e_1,e_n) \longmapsto \left(1,\epsilon(V)\right).
\]
Then by Theorem \ref{T:AC-L}, $\Lambda=\Theta\left(\omega_{V,\mu}[\chi]\right)$ corresponds to the character $\nu$ of $A_{\theta\left(\Psi_{M,\mu}\right)}$ such that
\[
    \nu: (e_1,e_n)\longmapsto \left(\epsilon\left(\frac{1}{2},\chi\cdot\mu^{-n}\cdot\chi_W^{-1},\psi_{E,\delta}\right),\, \epsilon(V)\cdot\epsilon\left(\frac{1}{2},\mu\cdot\chi_W^{-1}\boxtimes S_n, \psi_{E,\delta}\right)\right).
\]
The evaluation of $\nu$ at $e_{n-1}$ can be determined by its evaluation at $(e_1,e_n)$ and the sign of $\calR W$. To be more precise, $\nu$ takes $e_{n-1}$ to 
\begin{align*}
    &\epsilon(V)\cdot\epsilon\left(\calR W\right)\cdot\epsilon\left(\frac{1}{2},\chi\cdot\mu^{-n}\cdot\chi_W^{-1},\psi_{E,\delta}\right)\cdot\epsilon\left(\frac{1}{2},\mu\cdot\chi_W^{-1}\boxtimes S_n, \psi_{E,\delta}\right) \\
    =&\epsilon(V)\cdot\epsilon\left(\calR W\right)\cdot\epsilon\left(\frac{1}{2},\chi^\flat\cdot\mu^{-n},\psi_{E,\delta}\right)\cdot\epsilon\left(\frac{1}{2},\As^+_{L/E}(M_{1})\cdot\mu^{-1}, \psi_{E,\delta}\right)\cdot\omega_{E/F}(-1).
\end{align*}
Hence, if we let $V^\flat$ be the $n$-dimensional skew-Hermitian space as in (\ref{Eq-SignFlat}), then again by Theorem \ref{T:AC-L}, one can check that:
\begin{itemize}
    \item $\omega_{V^\flat,\mu}[\chi^\flat]\in\Pi_{\Psi^\flat}(\U(V^\flat))$ corresponding to the character $\nu_{n}$ of $A_{\Psi^\flat}$ such that
        \[
            \nu_{n}: (e_1, e_{n-1})\longmapsto \left(\epsilon\left(\frac{1}{2},\chi^\flat\cdot\mu^{-n},\psi_{E,\delta}\right),\, \epsilon\left(V^\flat\right)\cdot\epsilon\left(\frac{1}{2},\chi^\flat\cdot\mu^{-n},\psi_{E,\delta}\right)\right);
        \]

    \item the theta lift of $\omega_{V^\flat,\mu}[\chi^\flat]$ to $\U(\calR W)$ is non-zero and exactly equal to $\Lambda$.
\end{itemize}
These complete the proof of the case when $n$ odd.\\

Similarly, when $n$ is even, $\omega_{V,\mu}[\chi]\in\Pi_{\Psi_{M,\mu}}(\U(V))$ corresponds to 
\[
    \nu_{n+1}: (e_1,e_n) \longmapsto \left(\epsilon\left(\frac{1}{2},\chi\mu^{-n-1},\psi_{E,\delta}\right),\epsilon(V)\epsilon\left(\frac{1}{2},\chi\mu^{-n-1},\psi_{E,\delta}\right)\right).
\]
By Theorem \ref{T:AC-L}, $\Lambda$ corresponds to $\nu\in\Irr A_{\theta\left(\Psi_{M,\mu}\right)}$ such that $\nu~\big|_{A_{\Psi_{M,\mu}}}=\nu_{n+1}$, so
\[
    \nu(e_{n-1})= \epsilon(V)\cdot\epsilon\left(\calR W\right).
\]
Then again one can appeal to Theorem \ref{T:AC-L} to show that the theta lift of $\omega_{V^\flat,\mu}[\chi^\flat]$ is exactly $\Lambda$.

\end{proof}

\vskip 5pt

With this proposition in hand, we get
\begin{equation}\label{E:SS2}\tag{$\maltese.2$}
    m_V(\Pi,\mu) = \dim\Hom_{\U(W)}(\Lambda,\Sigma) = m_{V^\flat}(\Pi^\flat,\mu)
\end{equation}
is non-zero. In particular, $\Pi^\flat$ is non-zero. By \cite[Sect. 4.4(P1)]{MR3573972}, we know that:
\begin{itemize}
    \item The sign of the Hermitian space $W$ is given by
\[
    \epsilon(W)=\epsilon\left(\frac{1}{2}, M_0\cdot {^{\tau}M_{1}}\cdot\mu^{-1}\circ\Nm_{L/E}, \,\psi_{L,\delta}\right),
\]
where $\psi_{L,\delta}=\psi_F\left(\Tr_{L/F}(\delta\cdot~)\right)$.
\vskip 5pt

    \item $\Pi^\flat$ is an irreducible tempered representation has L-parameter $M^\flat=M'_0$ and corresponds to $\eta^\flat$, where
\[
    M'_0=M_0\cdot {^{\tau}M_{1}}\cdot M_{1}^{-1}\cdot\Upsilon\cdot\mu^{-1}\circ\Nm_{L/E},
\]
    and
        \[
            \eta^\flat(a_i)/\eta(a_i) = \epsilon\left(\frac{1}{2}, M_i\cdot{^{\tau}M_{1}}\cdot\mu^{-1}\circ\Nm_{L/E}, \,\psi_{L,\delta}\right)
        \]
    for all $i\geq 2$. 
\end{itemize}
Also note that
\[
    \epsilon\left(\calR W\right)=\epsilon(W)\cdot\omega_{K/F}\left(\delta^2\right)^n\cdot\omega_{E/F}(-1)^n.
\]
Substitute these into (\ref{Eq-SignFlat}), we get
\begin{equation*}\label{Eq-SignFlat2}\tag{$\dagger\dagger$}
    \epsilon\left(V^\flat\right)=\epsilon(V)\cdot\epsilon \left(\frac{1}{2}, \Ind_L^E\left(^{\tau}M_{1} \otimes (M/M_{1})\right)\cdot\mu^{-1}, \psi_{E,\delta}\right)\cdot\epsilon\left(\frac{1}{2},\As^+_{L/E}(M_{1})\cdot\mu^{-1}, \psi_{E,\delta}\right)\cdot\omega_{K/F}\left(\delta^2\right)^n.
\end{equation*}

\vskip 5pt

Now if we assume that Conjecture \ref{C-TGGP} holds for the L-parameter $M^\flat$, then it follows that:
\begin{enumerate}
    \item The multiplicity $m_{V}(\Pi,\mu)=1$. \\

    \item $V$ is the unique $(n+1)$-dimensional Hermitian space over $E$ predicted by the formula in Conjecture \ref{C-TGGP}(3). Indeed, note that for any semi-simple representation $N$ and any character $\calX$ of $WD_L$, we have 
\[
    \As^+(N\cdot\calX)=\As^+(N)\cdot\left(\calX~\Big|_{E^\times}\right).
\]
    Combining this with Conjecture \ref{C-TGGP}(3) for $M^\flat$, we know that
        \begin{align*}
            \epsilon\left(V^\flat\right)&=\epsilon\left(\frac{1}{2}, \As^+_{L/E}\left(M^\flat\right)\otimes\mu^{-1}, \psi_{E,\delta}\right)\cdot\omega_{K/F}\left(\delta^2\right)^{n(n-1)/2}\\
            &=\epsilon\left(\frac{1}{2}, \As^+_{L/E}\left(M_0\right)\otimes\mu^{-1}, \psi_{E,\delta}\right)\cdot\omega_{K/F}\left(\delta^2\right)^{n(n-1)/2}.
        \end{align*}
    Then applying the equality (\ref{Eq-SignFlat2}), we get
        \begin{align*}
            \epsilon(V)&=\epsilon\left(V^\flat\right)\cdot\epsilon \left(\frac{1}{2}, \Ind_L^E\left(^{\tau}M_{1} \otimes (M/M_{1})\right)\cdot\mu^{-1}, \psi_{E,\delta}\right)\cdot\epsilon\left(\frac{1}{2},\As^+_{L/E}(M_{1})\cdot\mu^{-1}, \psi_{E,\delta}\right)\cdot\omega_{K/F}\left(\delta^2\right)^n\\
            &=\epsilon\left(\frac{1}{2}, \As^+_{L/E}(M)\otimes\mu^{-1}, \psi_{E,\delta}\right)\cdot\omega_{K/F}\left(\delta^2\right)^{n(n+1)/2}.
        \end{align*}

    \item $\Pi$ is the unique member in $\Pi_{M}$ predicted by the formula in Conjecture \ref{C-TGGP}(4). Similar to (2), it follows from Conjecture \ref{C-TGGP}(4) that
    \begin{align*}
            \eta^\flat(a_i) &= \epsilon \left(\frac{1}{2}, \Ind_L^E\left(^{\tau}M_i \otimes (M_0/M_i)\right)\cdot\mu^{-1}, \psi_{E,\delta}\right)\\
            &= \epsilon\left(\frac{1}{2}, {^{\tau}M_i} \otimes (M_0/M_i)\cdot\mu^{-1}\circ\Nm_{L/E}, \psi_{L,\delta}\right)
        \end{align*}
        for all $i\geq 2$. Hence 
        \begin{align*}
            \eta(a_i)&=\eta^\flat(a_i)\cdot\epsilon\left(\frac{1}{2}, M_i\cdot{^{\tau}M_{1}}\cdot\mu^{-1}\circ\Nm_{L/E}, \,\psi_{L,\delta}\right)\\
            &=\epsilon\left(\frac{1}{2}, {^{\tau}M_i} \otimes (M/M_i)\cdot\mu^{-1}\circ\Nm_{L/E}, \,\psi_{L,\delta}\right)\\
            &=\epsilon \left(\frac{1}{2}, \Ind_L^E\left(^{\tau}M_i \otimes (M/M_i)\right)\cdot\mu^{-1}, \psi_{E,\delta}\right)
        \end{align*}
        for all $i\geq 2$. On the other hand, recall that $\eta(a_1)=\epsilon(W)$. This implies the desired equality
        \begin{align*}
            \eta(a_1)&=\epsilon\left(\frac{1}{2}, M_0\cdot {^{\tau}M_{1}}\cdot\mu^{-1}\circ\Nm_{L/E}, \,\psi_{L,\delta}\right) \\
            &=\epsilon \left(\frac{1}{2}, \Ind_L^E\left(^{\tau}M_1 \otimes (M/M_1)\right)\cdot\mu^{-1}, \psi_{E,\delta}\right).
        \end{align*}
\end{enumerate}
\vskip 5pt
The computation above shows that there is at most one $\Pi$ in the L-packet $\Pi_M$ such that $m_V(\Pi,\mu)\ne 0$.\\

\subsection{Reversing two seesaw diagrams: Existence}
Conversely, still under the assumption that Conjecture \ref{C-TGGP} holds for the L-parameter $M^\flat$, we can produce an irreducible tempered representation $\Pi'\in\Pi_M$ such that
\[
    m_{V'}(\Pi',\mu)\ne 0 
\]
for some $(n+1)$-dimensional skew-Hermitian space $V'$, from the unique irreducible tempered representation $\Pi^\flat\in\Pi_{M^\flat}$ such that
\[
    m_{V^\flat}(\Pi^\flat,\mu) \ne 0.
\] 
We do it by applying the two seesaw diagrams reversely as follows. First consider an analog of the seesaw diagram (\ref{SeeSaw-2}) (using the same splitting characters):
\begin{equation*}
\begindc{\commdiag}[8]
\obj(60,40)[aa]{$\U\left(\calR W'\right)$}
\obj(54,36)[a]{$~$}
\obj(60,-40)[b]{$\U\left(W'\right)$}
\obj(-60,40)[c]{$\U\left(V^\flat_K\right)$}
\obj(-60,-40)[d]{$\U(V^\flat)$}
\mor{aa}{b}{}[\atleft,\solidline]
\mor{a}{d}{}[\atleft,\solidline]
\mor{c}{d}{}[\atleft,\solidline]
\mor{c}{b}{}[\atleft,\solidline]
\enddc
\qquad \qquad
\begindc{\commdiag}[8]
\obj(60,40)[a]{${\Lambda'=\Theta\left(\omega_{V^\flat,\mu}[\chi^\flat]\right)}$}
\obj(54,36)[aa]{$~$}
\obj(60,-40)[b]{$\Sigma'$}
\obj(54,-36)[bb]{$~$}
\obj(-60,40)[cc]{$\Pi^\flat=\Theta(\Sigma')$}
\obj(-54,36)[c]{$~$}
\obj(-60,-40)[d]{$\omega_{V^\flat,\mu}[\chi^\flat]$}
\obj(-54,-36)[dd]{$~$}
\mor{a}{b}{}[\atleft,\solidline]
\mor{aa}{dd}{}[\atleft,\solidline]
\mor{cc}{d}{}[\atleft,\solidline]
\mor{c}{bb}{}[\atleft,\solidline]
\enddc
\end{equation*} 
where $W'$ is the unique $n$-dimensional Hermitian space over $L$ chosen by the theta dichotomy \cite[Sect. 4.4(P1)]{MR3573972}; that is, the theta lift $\Sigma'$ of $\Pi^\flat$ to $\U(W')$ is non-zero. Symmetrically, $\Pi^\flat=\Theta\left(\Sigma'\right)$ is the big theta lift of $\Sigma'$. By the seesaw identity, we have
\begin{equation*}\label{E:SS-3}\tag{$\maltese.3$}
    m_{V^\flat}\left(\Pi^\flat,\mu\right)=\dim\Hom_{\U(W')}\left(\Lambda',\Sigma'\right).
\end{equation*}
In particular, $\Lambda'$ is non-zero. It then follows from Theorem \ref{T:AC-P} and Theorem \ref{T:AC-L} that $\Lambda'$ is an irreducible unitary representation lies in the A-packet $\Pi_{\theta(\Psi_{M,\mu})}\left(\U\left(\calR W'\right)\right)$, where
\begin{align*}
            \theta\left(\Psi_{M,\mu}\right) &= \chi\cdot\mu^{-n}\cdot\chi_W^{-1}\chi_V + \chi_V\boxtimes S_{n-1} + \mu\cdot\chi_W^{-1}\chi_V\boxtimes S_n.
\end{align*}

\vskip 10pt

Next we shall use an analog of the seesaw diagram (\ref{SeeSaw-1}). The following is an analog of the key Proposition \ref{P:AdamsWR}.

\begin{Prop}\label{P:AdamsWR-2}
Let $V'$ be the $(n+1)$-dimensional skew-Hermitian space of sign
\begin{equation}\label{Eq-SignPrime}\tag{$\ddag$}
    \epsilon\left(V^\prime\right)=
    \epsilon\left(V^\flat\right)\cdot\epsilon\left(\calR W'\right)\cdot\epsilon\left(\frac{1}{2},\As^+_{L/E}(M_{1})\cdot\mu^{-1}, \psi_{E,\delta}\right)\cdot\omega_{E/F}(-1)^n 
\end{equation}
and
\[
    \omega'=\omega_{V^\prime,\mu}[\chi].
\]
Then $\Lambda'$ is the (big) theta lift of $\omega'$ to $\U(\calR W')$, i.e. $\Lambda'=\Theta(\omega')$. Here we are using the same splitting characters as described in (\ref{SeeSaw-1}).
\end{Prop}

\vskip 5pt

\begin{proof}
We first check the special case that $n=2$ and $\chi^\flat=\mu^2$. So
\[
    \Psi_{M,\mu}=\chi_W+\mu\boxtimes S_2
\]
and 
\[
    \theta\left(\Psi_{M,\mu}\right)=\chi_V+\mu\cdot\chi_W^{-1}\chi_V\boxtimes S_2+\left(\chi_V^c\right)^\vee.
\]
Let $V''$ and $\calR W''$ be the companion spaces of $V'$ and $\calR W'$ respectively. Consider the following map given by the theta correspondence:
\begin{equation*}
\theta_{4}:\Irr \U(V')\sqcup \Irr \U(V'')\longrightarrow \Irr \U(\calR W')\sqcup \Irr \U(\calR W''),
\end{equation*}
where:
\begin{equation*}
\pi\mapsto\begin{cases}
                         \theta_{\calR W'}(\pi)\in\Irr\left(\U(\calR W')\right) \quad &\textit{if } \theta_{\calR W'}(\pi)\ne 0;\\
                         \theta_{\calR W''}(\pi)\in\Irr\left(\U(\calR W'')\right) \quad &\textit{otherwise}.

                   \end{cases}
\end{equation*}
\vskip 5pt
\noindent This map is well-defined by the theta dichotomy. Using Theorem \ref{T:AC-P}, One can easily check by hand that this map restricts to a bijection
\[
    \theta_4: \Pi_{\Psi_{M,\mu}} (\U(V'))\sqcup \Pi_{\Psi_{M,\mu}} (\U(V''))\longrightarrow \Pi_{\theta\left(\Psi_{M,\mu}\right)} (\U(\calR W'))\sqcup \Pi_{\theta\left(\Psi_{M,\mu}\right)} (\U(\calR W'')).
\]
Hence $\Lambda'$ is the (big) theta lift of some
\[
    \omega'\in \Pi_{\Psi_{M,\mu}} (\U(V'))\sqcup \Pi_{\Psi_{M,\mu}} (\U(V'')).
\]
To check that $\omega'=\omega_{V',\mu}[\chi]$, one can use Theorem \ref{T:AC-L} and Lemma \ref{L:WR-L} to compute the character $\eta'\in\Irr A_{\Psi_{M,\mu}}$ associated to $\omega'$. We omit the details.\\

Excluding the special case above, the theta correspondence between $\U(V')\times \U(\calR W')$ is in the situation of Theorem \ref{T:AC-P}(2). It follows that the theta lift $\Theta(\omega_{V',\mu}[\chi])$ to $\U(\calR W')$ is non-vanishing. So the proof of the general case comes down to a computation of the labellings similar to the proof of Proposition \ref{P:AdamsWR}. We shall not repeat the tedious computation here. 

\end{proof}

\vskip 5pt

Now we can consider the following seesaw diagram, with respect to the same splitting characters as described in (\ref{SeeSaw-1}):
\begin{equation*}
\begindc{\commdiag}[8]
\obj(-60,40)[aa]{$\U\left(\calR W'\right)$}
\obj(-51,34)[a]{$~$}
\obj(-60,-40)[b]{$\U\left(W'\right)$}
\obj(60,40)[c]{$\U\left(V'_K\right)$}
\obj(60,-40)[d]{$\U(V')$}
\mor{aa}{b}{}[\atleft,\solidline]
\mor{a}{d}{}[\atleft,\solidline]
\mor{c}{d}{}[\atleft,\solidline]
\mor{c}{b}{}[\atleft,\solidline]
\enddc
\qquad \qquad
\begindc{\commdiag}[8]
\obj(-60,40)[aa]{$\Lambda'=\Theta(\omega_{V',\mu}[\chi])$}
\obj(-51,34)[a]{$~$}
\obj(-60,-40)[bb]{$\Sigma'$}
\obj(-54,-36)[b]{$~$}
\obj(60,40)[c]{$\Pi'=\Theta(\Sigma')$}
\obj(51,34)[cc]{$~$}
\obj(60,-40)[d]{$\omega_{V',\mu}[\chi]$}
\mor{aa}{bb}{}[\atleft,\solidline]
\mor{a}{d}{}[\atleft,\solidline]
\mor{c}{d}{}[\atleft,\solidline]
\mor{cc}{b}{}[\atleft,\solidline]
\enddc
\end{equation*}
Again by the seesaw identity, we have
\begin{equation}\label{E:SS4}\tag{$\maltese.4$}
    m_{V'}(\Pi',\mu) = \dim\Hom_{\U(W)}(\Lambda',\Sigma') = m_{V^\flat}(\Pi^\flat,\mu)
\end{equation}
is non-zero. In particular, $\Pi'$ is non-zero. By \cite[Sect. 4.6(P2)]{MR3573972}, we know that $\Pi'$ is an irreducible tempered representation of $\U(V'_K)$ lies in the L-packet $\Pi_M$. The construction above shows the existence of $\Pi'\in\Pi_M$ such that $m_{V'}(\Pi',\mu)\ne 0$.

\vskip 10pt

\subsection{Conclusion}
In summary, we have shown that: 

\begin{Prop}
Let $V_0$ be an $n$-dimensional Hermitian space over $E$, and $M'_0$ a tempered L-parameter for the unitary group $\U(V_{0,K})$. Assume that Conjecture \ref{C-TGGP} holds for the L-parameter $M'_0$. Then it also holds for the L-parameter of the form 
\[
    M = M'_0 \cdot\calX + M_1,
\]
where $\calX$ is any conjugate symplectic character of $L^\times$, and $M_1$ is any conjugate self-dual character of $L^\times$ of parity $(-1)^n$.
\end{Prop}

\vskip 5pt

\begin{proof}
As we have explicated above, given such an L-parameter $M$, one can construct an L-parameter $M^\flat$ of $\U(V_{0,K})$. As long as Conjecture \ref{C-TGGP} holds for the L-parameter $M^\flat$, it also holds for $M$. On the other hand, from the construction of $M^\flat$, one can see that 
\[
    M^\flat = M'_0\cdot\calY
\]
for some conjugate orthogonal character $\calY$ of $L^\times$. Thus by Lemma \ref{L:CT}, Conjecture \ref{C-TGGP} holds for $M^\flat$.

\end{proof}

\vskip 5pt

\begin{Cor}
Theorem \ref{T:Main} holds.
\end{Cor}

\vskip 5pt

\begin{proof}
Simply note that if $M$ is a summation of conjugate self-dual characters as described in Theorem \ref{T:Main}, then so is $M^\flat$.

\end{proof}

\vskip 5pt

The reader may notice the similarity of our set up with the paper \cite{X} of Hang Xue, in which he showed the Bessel case of the local GGP conjecture for unitary groups over $\mathbb{R}$. There, he worked also with L-parameters $M$ of the same form as those in Theorem \ref{T:Main}. Indeed, we are partly inspired by his results to consider these $M$'s. However, the inductive argument in our proof is different from that in \cite{X} (not to mention that the setting of our result is different).

\vskip 5pt

We end up this paper with a remark on the global conjecture \cite[Conj. 9.1]{TGGP}. One can expect to prove the global conjecture for the near equivalence class
\[
    M=M_1+\cdots+M_n
\]
with each $M_i$ conjugate self-dual automorphic character of $\GL_1$ of parity $(-1)^{n-1}$, by using the same argument. Instead of the Adams' conjecture used in this paper, one will need to show an analog of the Siegel-Weil formula in the global case, so that one can compare the theta integrals of $\omega_{V,\mu}$ and $\omega_{V^\flat,\mu}$. More precisely, let $\Omega_{V}$ and $\Omega_{V^\flat}$ be the Weil representation associated to $\U(V)\times\U(\calR W)$ and $\U(V^\flat)\times\U(\calR W)$ respectively, one needs to compare
\[
    \int_{[\U(V)]}\theta_\varphi(g,h)f(g)\,dg \quad \textit{for $\varphi\in\Omega_V$, $f\in\omega_{V,\mu}$, $g\in\U(V)$, $h\in\U(\calR W)$}
\]
and
\[
    \int_{[\U(V^\flat)]}\theta_{\varphi'}(g',h)f'(g')\,dg' \quad \textit{for $\varphi'\in\Omega_{V^\flat}$, $f'\in\omega_{V^\flat,\mu}$, $g'\in\U(V^\flat)$, $h\in\U(\calR W)$}.
\]
Unfortunately,  these theta integrals diverge in general. So one has to properly regularize these theta integrals first. Once a global analog of Proposition \ref{P:AdamsWR} has been established, the remaining parts should go over smoothly.
\vskip 10pt

\noindent{\bf Acknowledgments}: Both authors were supported by a Singapore government MOE Tier One grant R-146-000-320-114 during the course of this work. The work was completed when the second author visited the Erwin Schrodinger Institute in Vienna in April 2022; he thanks the ESI for excellent working conditions and inspiring atmosphere. The authors thank
Petar Baki\'c, Marcela Hanzer and Jialiang Zou for helpful discussion and comments. 

\vskip 15pt

\bibliographystyle{alpha}
\bibliography{TGGP4TRef}

\end{document}